\documentclass[a4paper,reqno,final]{amsart}

\usepackage{amsmath,amssymb,amsthm}
\usepackage{fullpage,enumitem}
\usepackage{showkeys}
\usepackage{mathtools}
\mathtoolsset{showonlyrefs,showmanualtags}
\RequirePackage[l2tabu,orthodox]{nag}
\usepackage[all, warning]{onlyamsmath}
\hypersetup{bookmarks=false,draft=false,breaklinks,colorlinks,dvipdfmx}

\numberwithin{equation}{section}

\theoremstyle{definition}
\newtheorem{thm}{Theorem}[section]

\newtheorem{dfn}[thm]{Definition}

\newenvironment{Ack}%
{\par \vspace{\baselineskip}%
 \noindent \textbf{Acknowledgements.}}%
{\par \vspace{\baselineskip}}%

\newcommand{\wt}{\widetilde}
\newcommand{\bs}{\backslash}
\newcommand{\midd}{\, \vert \,}

\newcommand{\lto}{\longrightarrow}

\newcommand{\lmto}{\longmapsto}

\newcommand{\lsto}{\xrightarrow{\ \sim \ }}

\newcommand{\bbC}{\mathbb{C}}
\newcommand{\bbF}{\mathbb{F}}
\newcommand{\bbN}{\mathbb{N}}
\newcommand{\bbP}{\mathbb{P}}

\newcommand{\bbZ}{\mathbb{Z}}
\newcommand{\bfA}{\mathbf{A}}
\newcommand{\bfH}{\mathbf{H}}
\newcommand{\bfK}{\mathbf{K}}
\newcommand{\frS}{\mathfrak{S}}
\newcommand{\fsl}{\mathfrak{sl}}

\DeclareMathOperator{\id}{id}
\DeclareMathOperator{\tr}{tr}
\DeclareMathOperator{\GL}{GL}
\DeclareMathOperator{\SU}{SU}
\DeclareMathOperator{\TF}{TF}
\DeclareMathOperator{\ind}{ind}

\DeclareMathOperator{\End}{End}
\DeclareMathOperator{\Hom}{Hom}
\DeclareMathOperator{\Ind}{Ind}
\DeclareMathOperator{\Irr}{Irr}

\DeclareMathOperator{\sgn}{sgn}

\newcommand{\fl}[1]{\lfloor #1 \rfloor}
\newcommand{\hg}[2]{{}_{#1} F_{#2}}
\newcommand{\abs}[1]{\left| #1 \right|}
\newcommand{\gen}[1]{\langle #1 \rangle}
\newcommand{\rst}[2]{\left. #1 \right|_{#2}}

\newcommand{\pair}[2]{\left< #1,#2 \right>}

\begin{document}

\title{Gelfand pairs and spherical functions for Iwahori-Hecke algebras}
\author{Shintarou Yanagida}
\date{2021.04.25}
\keywords{Iwahori-Hecke algebras, Gelfand pairs, zonal spherical functions}
\thanks{This work is supported by supported by JSPS KAKENHI Grant Number 19K03399, %
 and also by the JSPS Bilateral Program %
 ``Elliptic algebras, vertex operators and link invariant".}
\address{Graduate School of Mathematics, Nagoya University
Furocho, Chikusaku, Nagoya, 464-8602, Japan.}
\email{yanagida@math.nagoya-u.ac.jp}


\begin{abstract}
We introduce the notion of Gelfand pairs and zonal spherical functions for 
Iwahori-Hecke algebras.
\end{abstract}

\maketitle

\section{introduction}\label{s:intro}

This short note is a spin off of the collaboration work \cite{HHY} 
with Masahito Hayashi and Akihito Hora, and it was originally planned to be included
as Appendix D of that paper.
The aim is to give an analogue of the theory of spherical functions
for Iwahori-Hecke algebras.
To explain the motivation, let us briefly explain the parts of \cite{HHY}
which concern this note.

In \cite{HHY}, a certain discrete probability distribution is introduced,
which originates in the irreducible decomposition of the $\SU(2)$-$\frS_n$-bimodule 
\begin{align}\label{eq:intro:SW}
 (\bbC^2)^{\otimes n} = \sum_{x=0}^{\fl{n/2}} U_{(n-x,x)} \boxtimes V_{(n-x,x)},
\end{align}
i.e., the tensor space appearing in the classical Schur-Weyl duality of 
the special unitary group $\SU(2)$ and the symmetric group $\frS_n$.
Here we denoted by $V_{(n-x,x)}$ the irreducible $\frS_n$-representation 
associated to the partition $(n-x,x)$, 
and by $U_{(n-x,x)}$ the corresponding irreducible $\SU(2)$-representation.
The distribution has four non-negative integer parameters $n,m,k,l$,
and the probability mass function is denoted by $p(x)=p(x \midd n,m,k,l)$
for $x=0,1,\dotsc,\fl{n/2}$.

The function $p(x)$ has surprisingly rich property. For example, 
it can be written as a summation of Hahn polynomials \cite[Theorem 2.2.1]{HHY}.
Hahn polynomials are $\hg{3}{2}$-hypergeometric orthogonal polynomials,
and appear in the computation as values of zonal spherical functions 
for the Gelfand pair $(\frS_n, \frS_m \times \frS_{n-m})$.
In addition, the function $p(x)$ can be written by a single Racah polynomial 
\cite[Theorem 2.2.2]{HHY}, which is a $\hg{4}{3}$-hypergeometric orthogonal polynomial 
sitting in the first line of the Askey scheme.
Moreover, the cumulative distribution function $\sum_{u=0}^x p(u)$ 
can also be written by a terminating $\hg{4}{3}$-hypergeometric series 
\cite[Theorem 2.2.4]{HHY}.
In the course of deriving these theorems, 
some hypergeometric summation formulas are also obtained 
\cite[Corollaries 2.2.3 and 2.2.5]{HHY}.

These arguments have natural $q$-analogues, which are discussed in \cite[Appendix C]{HHY}.
There is introduced a $q$-analogue of $p(x)$ \cite[(C.2.11)]{HHY}, 
and are obtained the corresponding basic hypergeometric summation formulas 
\cite[Corollaries C.3.2 and C.3.4]{HHY}.
Although we may say that those computations are natural $q$-analogues,
the setting is a little bit artificial in the viewpoint of representation theory.
In \cite[Appendix C]{HHY}, instead of considering 
the $n$-tensor space \eqref{eq:intro:SW}, we consider 
\begin{align}
 (\bbC^2)^{\bbP^1(\bbF_q^n)} = (\bbC^2)^{[n]_q},	
\end{align}
where $\bbF_q$ denotes the finite field of order $q$, 
$\bbP^1(\bbF_q^n)$ denotes the projective line for the linear space $\bbF_q^n$,
and $[n]_q=1+q+\dotsb+q^{n-1}$ is the $q$-integer.
In other words, we take the formula $\abs{\bbP^1(\bbF_q^n)}=[n]_q$ of the number of points
as a clue of $q$-deformation.
As for the group action, 
instead of the permutation action of $\frS_n$ on $(\bbC^2)^{\otimes n}$, 
we consider the action of the Chevalley group $\GL(n,\bbF_q)$ 
on $(\bbC^2)^{\bbP^1(\bbF_q^n)}$ which is induced by 
the natural action of that on $\bbF_q^n$.
The usage of Hahn polynomials as zonal spherical functions 
for the Gelfand pair $(\frS_n, \frS_m \times \frS_{n-m})$
is then replaced by $q$-Hahn polynomials as zonal spherical functions 
for the pair $(\GL(n,\bbF_q), P(m,n-m,\bbF_q))$,
where $ P(m,n-m,\bbF_q)$ denotes the maximal parabolic subgroup of 
$\GL(n,\bbF_q)$ with block sizes $m$ and $n-m$.

As mentioned at the end of the introduction of \cite[Appendix C]{HHY},
we have another candidate of $q$-analogue for \eqref{eq:intro:SW}:
the $q$-Schur-Weyl duality between 
the quantum group $U_{\sqrt{q}}(\fsl_2)$ and the Iwahori-Hecke algebra $H_{q}(\frS_n)$
acting on the same tensor space $(\bbC^2)^{\otimes n}$ as \eqref{eq:intro:SW}, 
discovered by Jimbo \cite{J}.
This one looks more natural in the representation theoretic viewpoint.

However, there are several missing items in this $q$-Schur-Weyl duality to
make a $q$-analogue of the arguments of \cite{HHY}.
One of them is 
the theory of Gelfand pairs and spherical functions for Iwahori-Hecke algebras.
As sketched above, we used in \cite{HHY} the corresponding theory for finite groups,
for which we refer \cite[\S 7.1]{M}.
Since Iwahori-Hecke algebras are deformation of group algebras of Coxeter groups, 
it is natural to expect such a theory exists,
although we could not find it in literature.

The purpose of this note is to give the start line of the theory of Gelfand pairs
and zonal spherical functions for Iwahori-Hecke algebras, 
which will be demonstrated in \S \ref{s:main}.
The main objects introduced are Gelfand pairs in Definition \ref{dfn:Gp}
and zonal spherical functions in Definition \ref{dfn:zsf}.
The main Theorem \ref{thm:zsf} gives fundamental properties of 
our zonal spherical functions, which can be regarded as 
a Hecke-analogue of the properties for finite groups \cite[VII.1, (1.4)]{M}.

Below is the list of notations used in the main text.
\begin{enumerate}[nosep]
\item
We denote by $\bbN := \bbZ_{\ge 0} = \{0,1,2,\ldots\}$ the set of non-negative integers.

\item
For a finite set $S$, we denote by $\abs{S}$ the number of elements in $S$.

\item
For a group $G$, we denote by $e \in G$ the unit element.
\end{enumerate}

\section{Recollection on Iwahori-Hecke algebra}\label{s:pr}

As a preliminary of the main \S \ref{s:main}, we recall several basic notions 
on the structure and representation theories of Hecke algebras.
Our main reference is \cite{GP}.
Hereafter, Iwahori-Hecke algebras will be just called Hecke algebras.

\subsection{Generic Hecke algebra}\label{ss:gH}

Here we explain our notation for the Iwahori-Hecke algebra,
borrowing terminology and symbols from \cite{GP}.
We make one major change on symbols: 
we denote the generic Hecke algebra by $\bfH$,
whereas it is denoted by $\bfK \bfH$ in \cite{GP}.

Let $W$ be a finite Coxeter group with generating set $S$.
We denote by $\ell(w)$ the length of a minimal expression of $w \in W$ with respect to $S$.
Let $A$ be a unital commutative ring, and take a parameter set 
$\{a_s \mid s \in S\} \subset A$ such that $a_s=a_t$ if $s,t \in S$ are conjugate in $W$.
We denote by 
\begin{align}
 H = H_A(W,S,\{a_s \mid s \in S\})
\end{align}
the Iwahori-Hecke algebra associated to the Coxeter system $(W,S)$ over $A$,
and simply call it the Hecke algebra.
It is a unital associative $A$-algebra 
with generators $\{T_s \mid s \in S\}$ and relations
\begin{align*}
\begin{cases}
       T_s^2 = (a_s-1)T_s+a_s  &(s \in S), \\
 (T_s T_t)^{m_{s t}} = 1 &(s,t \in S, \, s \neq t, \, m_{s t} < \infty).
\end{cases}
\end{align*}
The first relation is called the quadratic relation.
In the second relation we denoted by $m_{s,t}$ the order of the element $s t$ in $W$.
Let us recall some basic properties of $H$.
For the detail and proofs we refer \cite[4.4.3, 4.4.6]{GP}.
\begin{enumerate}[nosep]
\item 
For each $w \in W$, take a reduced expression $w=s_{i_1} \dotsm s_{i_l}$.
Then 
\begin{align}
 T_w := T_{s_{i_1}} \dotsm T_{s_{i_l}} \in H
\end{align}
is independent of the chose of the reduced expression.
In particular, we have $T_e=1$ for the unit $e \in W$.
Also the quadratic relation is rewritten as the factorized form $(T_s+1)(T_s-u)=0$.

\item
For $s \in S$ and $w \in W$, we have 
\begin{align}
 T_s T_w = 
 \begin{cases} 
  T_{s w} & (\ell(s w)>\ell(w)) \\ (a_s-1)T_{s w}+u T_w & (\ell(s w)<\ell(w))
 \end{cases}.
\end{align}

\item
As an $A$-module, $H$ is free and finitely generated, 
and has an $A$-basis $\{T_w \mid w \in W\}$.
\end{enumerate}
We always assume $a_s$ are invertible in $A$. 
Then the basis element $T_w$ is invertible in $H$ for any $w \in W$ by \cite[8.1.1]{GP}.

We recall two types of modules over the Hecke algebra $H$.
The first one is the index representation given by 
\begin{align}\label{eq:ind}
 \ind\colon H \lto A, \quad T_w \lmto a_w \quad (w \in W),
\end{align}
and the second one is the sign representation given by 
\begin{align}\label{eq:sgn}
 \sgn\colon H \lto A, \quad T_w \lmto (-1)^{\ell(w)} \quad (w \in W).
\end{align}
The well-definedness of these modules are shown by the quadratic relation 
$(T_s+1)(T_s-u)=0$.

Next we cite from \cite[\S 8.1]{GP} the specialization argument of the Hecke algebra.
Let $\{u_s \mid s \in S\}$ be a set of indeterminates over $\bbC$
such that $u_s=u_t$ for $s,t \in S$ conjugate in $W$.
We set  $\bfA := \bbZ[u^{\pm 1}]$.
The generic Hecke algebra $H_\bfA$ is defined to be the Hecke algebra over $\bfA$:
\begin{align}
 H_\bfA := H_{\bfA}(W,S,\{u_s \mid s \in S\}).
\end{align}
Now assume that there is  a ring map $\theta\colon \bfA \to \bbC$ with 
$\theta(u_s)=q \in \bbC \setminus \{0\}$, which will be called a specialization map.
The map $\theta$ induces a $\bbC$-algebra, which is denoted by
\begin{align}
 H_q := H_\bfA \otimes_{\bfA} \bbC.
\end{align}
Then by \cite[8.1.7]{GP}, there exists a finite Galois extension 
$\bfK \supset \bbC(u)$ such that the $\bfK$-algebra
\begin{align}
 \bfH := H_{\bfA} \otimes_{\bfA} \bfK
\end{align}
is semisimple and isomorphic to the group algebra $\bfK[W]$.
Abusing the terminology, we also call $\bfH$ the generic Hecke algebra.
Moreover, if the parameter $q \in \bbC \setminus \{0\}$ is taken such that $H_q$ 
is semisimple, then the specialization map $\bfH \to H_q$ induces a bijection
\begin{align}\label{eq:Irr:bij}
 \Irr(\bfH) \lsto \Irr(H_q)
\end{align}
between the sets of isomorphism classes of finite-dimensional simple modules 
over $\bfH$ and $H_q$ respectively.
In particular, for a semisimple $H_q$, the set $\Irr(H_q)$ is independent of $q$.

It is known that for a generic complex number $q$, the algebra $H_q$ is semisimple.
A particular case is $q=1$, where we have $H_1=\bbC[W]$,
and thus isomorphism classes of finite-dimensional simple $\bfH$-modules 
correspond bijectively to those of finite-dimensional irreducible representations of $W$.
Also note that the index representation \eqref{eq:ind} and 
the sign representation \eqref{eq:sgn} correspond respectively to the trivial and 
sign representations of the Coxeter group $W$ under the specialization $q=1$.

The generic Hecke algebra $\bfH$ over $\bfK$ is split semisimple 
in the sense of \cite[Chap.\ 7]{GP}.
Let us explain what it means.
Since $\bfH$ is a semisimple $\bfK$-algebra,
there is a direct sum decomposition $\bfH = \bigoplus_V \bfH(V)$,
where $V$ runs over the simple $\bfH$-modules, 
and each $\bfH(V)$ is a simple $\bfK$-algebra.
Thus $\bfH(V)$ is isomorphic to a matrix algebra $M_{n_V}(D_V)$,
where $n_V$ is the multiplicity of $V$ in the decomposition of $\bfH$ 
as a module over itself, and $D_V$ is a division algebra over $\bfK$.
We also have $D_V \simeq \End_{\bfH}(V)$ and $\dim_{\bfK}V = n_V \dim_{\bfK}D_V$.
Now every simple $\bfH$-module $V$ is split simple. i.e., $\dim_{\bfK} D_V=1$.
This is the defining condition of split semisimplicity.

Let us close this subsection by the case $W=\frS_n$ with
$S=\{s_1,\dotsc,s_{n-1}\}$ being the set of transpositions $s_i=(i,i+1)$.
All of the generators $s_i$ are conjugate in $W$, and 
the associated Hecke algebra $H$ has a unique parameter $q$.
Similarly, the generic Hall algebra $\bfH$  has a unique parameter $u$.
The Hecke algebra $H_q(\frS_n)$ mentioned in \S \ref{s:intro} 
is nothing but $H_{\bbC}(W,S,q)$.
As for the simple modules, the set $\Irr(\bfH)$ of 
finite-dimensional simple $\bfH$-modules sits in a series of bijections 
\begin{align}\label{eq:Irr=Par}
 \Irr(\bfH) \lsto \Irr(H_1) \lsto \{\lambda \vdash n\},
\end{align}
where $\lambda \vdash n$ means that $\lambda$ is a partition of $n$.
We will denote by $V_\lambda$ the simple $\bfH$-module 
corresponding to $\lambda \vdash n$ under \eqref{eq:Irr=Par}.
Note that the index representation \eqref{eq:ind} is $V_{(n)}$
and the sign representation $\sgn$ \eqref{eq:sgn} is $V_{(1^n)}$.

\subsection{Schur elements}\label{ss:c_V}

We continue to use the symbols in the previous subsection.
Thus $H$ denotes the Hecke algebra $H$ associated to 
a finite Coxeter system $(W,S)$ defined over a commutative ring $A$.

We recall trace functions and their relation to the center $Z(H)$ of $H$,
referring \cite[\S 7.1]{GP} for the detail.
An $A$-linear map $f \in H^* := \Hom_A(H,A)$ is called a trace function on $H$ if 
$f(h h')=f(h' h)$ for any $h, h' \in H$.
We denote the $A$-module of trace functions on $H$ by
\begin{align}
 \TF(H) := \{f\colon H \to A \mid \text{trace functions on $H$}\}.
\end{align}
$H$ is a symmetric algebra in the sense of \cite[7.1.1]{GP}, i.e., 
it is equipped with a trace function $\tau\colon H \to A$ such that
the $A$-bilinear form 
\begin{align}\label{eq:tau()}
 H \otimes_A H \lto A, \quad  h \otimes h' \lmto \tau(h h')
\end{align}
is non-degenerate.
Such a map $\tau$ is explicitly given by 
\begin{align}
 \tau\colon H \lto A, \quad \tau(T_e)=1, \quad \tau(T_w)=0 \quad (w \neq e).
\end{align}
We sometimes denote it by $\tau_H$ for the distinction. 

For $w \in W$ with a reduced expression $w=s_1 \dotsm s_l$, $s_i \in S$,
we set $a_w := a_{s_1} \dotsm a_{s_l} \in A$.
By \cite[8.1.1]{GP}, we have 
\begin{align}\label{eq:T^vee_w}
 \tau(T_w^\vee T_{w'}) = \delta_{w,w'}, \quad 
 T_w^{\vee} := a_w^{-1} T_{w^{-1}} \in H \quad (w,w' \in W).
\end{align}
Then for $f \in H^*$,  we set 
\begin{align}\label{eq:f^*}
 f^* := \sum_{w \in W} f(T_w) T_w^\vee \in H.
\end{align}
By \cite[7.1.7]{GP}, we have
\begin{align}\label{eq:TF=ZH}
 f \in \TF(H) \iff f^* \in Z(H).
\end{align}

An important class of trace functions is given by the character of an $H$-module.
Let us recall the definition.
Let $V$ be an $H$-module which is finitely generated and free as an $A$-module, 
and $\rho_V\colon H \to \End_K(V)$ be the corresponding $A$-algebra homomorphism.
The character of $V$ is defined to be the $A$-linear map 
\begin{align}
 \chi_V\colon H \lto A, \quad \chi_V(h) := \tr_V \rho_V(h),
\end{align}
where $\tr_V$ denotes the usual matrix trace.
By the property of matrix trace, 
the character is a trace function, i.e., $\chi_V \in \TF(H)$.

Next we introduce Schur elements, following \cite[\S 7.2]{GP}.
Hereafter we consider the Hecke algebra $H$ defined over a field $K$,
instead of a commutative ring $A$.
For $H$-modules $V,V'$ and $\varphi \in \Hom_K(V,V')$, 
we define $I(\varphi) \in \Hom_K(V,V')$ by 
\begin{align}
 I(\varphi)(v) := \sum_{w \in W}\varphi(v T_w) T_w^\vee \quad (v \in V).
\end{align}
By \cite[7.1.10]{GP}, we have $I(\varphi) \in \Hom_H(V,V')$.
Now assume that $V$ is a split simple $H$-module as explained in the previous subsection.
Then by \cite[7.2.1]{GP}, there is a unique element $c_V \in K$ such that 
\begin{align}
 I(\varphi) = c_V \tr_V(\varphi) \id_V
\end{align}
for any $\varphi \in \End_K(V)$, 
depending only on the isomorphism class of $V$.
The element $c_V$ is called the Schur element associated to $V$.

Let us explain representation theoretic meaning of Schur elements.
Let $V,V'$ be finite-dimensional simple $H$-modules. 
Then by \cite[7.2.2]{GP}, we have 
\begin{align}\label{eq:chi-orth}
 \sum_{w \in W} \chi_V(T_w) \chi_{V'}(T_w^\vee) = \delta_{V,V'} c_V \dim_K V ,
\end{align}
where $\delta_{V,V'}=1$ if $V \simeq V'$, and $\delta_{V,V'}=0$ otherwise.
Following \cite[8.1.8]{GP}, we define the Poincar\'e polynomial $P_W$ 
of the Coxeter group $W$ to be the Schur element $c_{\ind}$ 
associated to the index representation $\ind$ \eqref{eq:ind} 
of the generic Hecke algebra $\bfH$. 
By \eqref{eq:chi-orth} and \eqref{eq:T^vee_w}, it is given by 
\begin{align}\label{eq:P_W}
 P_W := c_{\ind} = \sum_{w \in W} \ind(T_w) \ind(T_w^\vee) 
      = \sum_{w \in W} \ind(T_w) \in \bbN[u_s \mid s \in S].
\end{align}

Schur elements are also related to idempotents of 
the generic Hecke algebra $\bfH$, as shown in \cite[7.2.7]{GP}.
Since $\bfH$ is split semisimple as explained in the previous \S \ref{ss:gH}, 
it has a decomposition
\begin{align}
 \bfH = \bigoplus_{V \in \Irr(\bfH)} \bfH(V)
\end{align}
into simple $\bfK$-algebras $\bfH(V)$, 
where $V$ runs over finite-dimensional simple $\bfH$-modules,  
and $\bfH(V) \simeq M_{n_V}(\bfK)$ with $n_V := \dim_{\bfK} V$.
We have the corresponding decomposition of the unit $1$ of $\bfH$:
\begin{align}
 1 = \sum_{V \in \Irr(\bfH)} e_V, \quad e_V \in \bfH(V).
\end{align}
Then $e_V$ are central primitive idempotents and mutually orthogonal, 
i.e., $e_V e_{V'} = 0$ if $V \not\simeq V'$.
We have $e_V \bfH \simeq V$ as $\bfH$-modules. 
Moreover, by \cite[7.2.6]{GP}, the Schur element $c_V$ is non-zero 
for any $V \in \Irr(\bfH)$, and using the character $\chi_V$ we have
\begin{align}\label{eq:e_V}
 e_V = \frac{1}{c_V} \sum_{w \in W} \chi_V(T_w) T_w^\vee.
\end{align}
In particular, using \eqref{eq:T^vee_w} and \eqref{eq:P_W}, we can compute the idempotent
$e_{\ind}$ corresponding to the index representation $\ind$ \eqref{eq:ind} as 
\begin{align}\label{eq:e_ind}
 e_{\ind} = \frac{1}{P_W}\sum_{w \in W} T_w. 
\end{align}

For a finite-dimensional simple $\bfH$-module $V$, let $\chi_V^* \in \bfH$ be 
the element associated to the character $\chi_V$ by the correspondence \eqref{eq:f^*}.
Since $\chi_V$ is a trace function, 
we have $\chi_V^* \in Z(\bfH)$ by the equivalence \eqref{eq:TF=ZH}.
Moreover, by \cite[7.2.8]{GP}, the elements 
$\{\chi_V^* \mid V \in \Irr(\bfH)\}$ form a $\bfK$-basis of $Z(\bfH)$, and
\begin{align}\label{eq:chi^*}
 \chi_V^* = c_V e_V.
\end{align}
The idempotence and orthogonality of $e_V$ yields 
\begin{align}\label{eq:chi^*-orth}
 \chi_V^* \chi_{V'}^* = \delta_{V,V'} c_V \chi_V^* \quad (V,V' \in \Irr(\bfH)).
\end{align}

Let us consider the behavior of objects discussed so far under the specialization
$\bfH \to H_{q=1}=\bbC[W]$ induced by $\bfK \to \bbC$, $u \mapsto 1$.
Obviously we have $T_w \mapsto w$ and $T_w^\vee \mapsto w^{-1}$ for each $w \in W$.
For a finite-dimensional $\bfH$-module $V$, 
we denote by $V_{q=1}$ the representation of $W$ induced by 
the bijection \eqref{eq:Irr:bij} of finite-dimensional simple modules.
Then by \cite[Example 7.2.5]{GP}, the specialization yields
\begin{align}
 c_V \lmto \frac{\abs{W}}{\dim_\bbC V_{q=1}}.
\end{align}
Thus we may say that the Schur element $c_V$ is a ``$q$-analogue" 
of the quantity $\abs{W}/\dim V_{q=1}$.
The relation \eqref{eq:chi-orth} reduces to 
the orthogonal relation for characters $\chi_{V_{q=1}}$ of the finite group $W$:
\begin{align}
 \frac{1}{\abs{W}}\sum_{w \in W} \chi_{V_{q=1}}(w) \chi_{V'_{q=1}}(w^{-1}) 
 = \delta_{V,V'}.
\end{align}
The Poincar\'e polynomial $P_W = c_{\ind}$ \eqref{eq:P_W} reduces to $\abs{W}$,
and the idempotent $e_V$ \eqref{eq:e_V} for a simple $\bfH$-module $V$ reduces to 
\begin{align}
 e_{V_{q=1}} = 
 \frac{\dim_\bbC V_{q=1}}{\abs{W}} \sum_{w \in W} \chi_{V_{q=1}}(w) w^{-1},
\end{align}
which is the primitive idempotent associated to 
the irreducible representation $V_{q=1}$ of $W$.
In particular, the idempotent $e_{\ind}$ for the index representation reduces to
$\abs{W}^{-1} \sum_{w \in W} w$, the idempotent of the group algebra $\bbC[W]$
corresponding to the trivial representation of $W$.

We close this subsection by an explicit formula of Schur elements in the case $W=\frS_n$.
Recall \eqref{eq:Irr=Par} that finite-dimensional simple $\bfH$-modules are 
parametrized by partitions of $n$.
Let $\lambda=(\lambda_1,\lambda_2,\dotsc)$ be such a partition.
We denote by $\ell=\ell(\lambda)$ the length of $\lambda$, i.e., 
the number of non-zero $\lambda_i$'s.
Let $V_\lambda$ be the simple $\bfH$-module corresponding to $\lambda$,
and $c_\lambda \in \bfK$ be the Schur element associated to $V_\lambda$.
By \cite[9.4.5, 10.5.1, 10.5.2]{GP}, we have 
\begin{align}\label{eq:c_V}
 c_\lambda^{-1} = u^{n(\lambda)} 
 \prod_{1 \le i<j<\ell} \bigl[\wt{\lambda}_i-\wt{\lambda}_j\bigr]_u \, \Big/ \, 
 \prod_{1 \le i \le \ell} \bigl[\wt{\lambda}_i\bigr]_u!
\end{align}
with $n(\lambda):=\sum_{i=1}^{\ell}(i-1)\lambda_i$ 
and $\wt{\lambda}_i := \lambda_i+\ell-i$.
We also used the symbols of $q$-integers and $q$-factorials:
\begin{align}
 [n]_q := 1+q+\dotsb+q^{n-1}, \quad [n]_q ! := [1]_q \, [2]_q \dotsm [n]_q.
\end{align}
Note that under the specialization $u \to 1$ we have 
\begin{align}
 \rst{c_\lambda^{-1}}{u \to 1} = 
 \prod_{1 \le i<j<\ell} \bigl(\wt{\lambda}_i-\wt{\lambda}_j\bigr) \, \Big/ \, 
 \prod_{1 \le i \le \ell} \wt{\lambda}_i ! = 
 \frac{\dim V_{\lambda,q=1}}{\abs{W}},
\end{align}
where the second equality is nothing but the hook length formula 
for the dimension of the irreducible representation $V_{\lambda,q=1}$ 
of $W=\frS_n$ corresponding to the partition $\lambda$.

The correspondence $f \mapsto f^*$ yields from \eqref{eq:tau()}
a non-degenerate symmetric bilinear form on $H^*$, which is denoted by 
\begin{align}\label{eq:pair_H*}
 \pair{f}{g}_{H^*} := \tau(f^* g^*) \quad (f,g \in H^*).
\end{align}
Then \eqref{eq:chi-orth} is rewritten by the pairing \eqref{eq:pair_H*} as 
\begin{align}
 \pair{\chi_V}{\chi_{V'}}_{H^*} = c_V \delta_{V,V'},
\end{align}

\section{Gelfand pairs and zoal spherical functions for Iwahori-Hecke algebras}
\label{s:main}

We give a Hecke analogue of the theory of Gelfand pairs and zonal spherical functions 
for finite groups.
We refer \cite[VII.1]{M} for the theory of finite groups. 
We continue to use the symbols in the previous \S \ref{s:pr}.
So $\bfH$ is the generic Hecke algebra associated to a finite Coxeter system $(W,S)$
defined over the finite Galois extension $\bfK \supset \bbC(u_s \mid s \in S)$.

For a subset $J \subset S$, we call the subgroup $W_J := \gen{J} \subset W$ 
generated by $J$ the parabolic subgroup, following the terminology in \cite[1.2]{GP}.
The corresponding subalgebra 
\begin{align}
 \bfH_J := \gen{T_w \mid w \in W_J}_\bfK \subset \bfH 
\end{align}
is called the parabolic subalgebra of $\bfH$ associated to $J$.

We denote by $e_J \in Z(\bfH_J)$ the idempotent \eqref{eq:e_ind} 
corresponding to the index representation of $\bfH_J$.
We regard it as an element of $\bfH$. 
Thus we have 
\begin{align}
 e_J = \frac{1}{P_J}\sum_{w \in W_J} T_w \in \bfH,
\end{align}
where $P_J:=P_{W_J} \in \bbN[u_s \mid s \in J] \subset \bfK$ 
denotes the Poincar\'e polynomial \eqref{eq:P_W} of $W_J$.

Now we consider the $\bfH$-module $e_J \bfH$.
It is isomorphic to the induced module 
\begin{align}
 \Ind_{\bfH_J}^{\bfH}(e_J \bfH_J) = e_J \bfH_J \otimes_{\bfH_J} \bfH
\end{align}
of the index representation $e_J \bfH_J$.
Since $e_J$ is an idempotent, we have a $\bfK$-algebra isomorphism
\begin{align}
 \End_{\bfH}(e_J \bfH) \lsto e_J \bfH e_J, \quad
 \varphi \lmto \varphi(e_J).
\end{align}

By \cite[9.1.9]{GP}, the multiplicities $m(V)$ of the decomposition 
\begin{align}
 e_J \bfH = \bigoplus_{V \in \Irr(\bfH)} V^{\oplus m(V)}
\end{align} 
into finite-dimensional simple $\bfH$-modules $V$ are preserved 
under the specialization $\bfH \to H_{q=1}$, $u_s \mapsto 1$.
In other words, the multiplicity $m(V)$ is the same as that
in the irreducible decomposition of the representation $1_{W_J}^W$ 
of the finite group $W$ obtained by inducing  
the trivial one-dimensional representation of the subgroup $W_J$.
Now let us recall the notion of Gelfand pair for finite groups \cite[VII.1, (1.1)]{M}:
A pair $(G,K)$ of a finite group $G$ and its subgroup $K$ is called a 
Gelfand pair of finite groups if the induced representation $1_K^G$ is multiplicity-free,
i.e., each of the non-zero multiplicities is one.
Thus the following definition seems to be natural.

\begin{dfn}\label{dfn:Gp}
Let $(W,S)$ be a finite Coxeter system and $J \subset S$ be a subset.
The pair $(S,J)$ or $(\bfH,\bfH_J)$ is called a Gelfand pair 
if the $\bfH$-module $e_J \bfH$ is multiplicity-free, i.e., 
the pair $(W_S,W_J)$ is a Gelfand pair of finite groups.
\end{dfn}

Hereafter we assume $(S,J)$ is a Gelfand pair, and denote by
\begin{align}
 e_J \bfH = \bigoplus_{i=1}^r V_i
\end{align}
the decomposition into simple $\bfH$-modules.
Let $\chi_i$ be the character of $V_i$ 
and $\chi_i^* \in \bfH$ be the element given by \eqref{eq:f^*}.
Then we introduce:

\begin{dfn}\label{dfn:zsf}
For $i=1,2,\dotsc,r$, we define an element $\omega_i \in \bfH$ by 
\begin{align}
 \omega_i := \chi_i^* e_J = e_J \chi_i^*,
\end{align}
where the equality follows from $\chi_i^* \in Z(\bfH)$,
explained in the paragraph of \eqref{eq:chi^*}.
We call $\omega_i$'s the zonal spherical functions of the Gelfand pair $(S,J)$.
\end{dfn}

The zonal spherical functions $\omega_i$ satisfy similar properties as
the zonal spherical functions in the standard sense.
The following lemma is a ``Hecke-analogue" of \cite[VII.1, (1.4)]{M}.

\begin{thm}\label{thm:zsf}
The elements $\omega_i$ ($i=1,2,\dotsc,r$) satisfy the following properties.
\begin{enumerate}[nosep]
\item
Let us express the expansion of $\omega_i$ in terms of the basis $\{T_w \mid w \in W\}$
as $\omega_i = \sum_{w \in W} \omega_i(w) T_w$.
Then $\omega_i(e) = 1$.

\item \label{i:prp:H-zph:2}
$\omega_i \omega_j = \delta_{i,j} c_i \omega_i$, where $c_i := c_{V_i} \in \bfK$
is the Schur element associated to $V_i$ (see \S \ref{ss:c_V}).

\item \label{i:prp:H-zph:3}
$\omega_i \in e_J \bfH e_J \cap V_i$.

\item
Zonal spherical functions $\omega_i$, $i =1,\dotsc,r$, 
form a $\bfK$-basis of $e_J H e_J$.
\end{enumerate}
\end{thm}

\begin{proof}
\begin{enumerate}[nosep]
\item
By definition we have 
$\omega_i = \left(P_J^{-1} \sum_{v \in W_J} T_v\right)
 \left(\sum_{w \in W}\chi_i(T_w) T_w^\vee\right)$.
Since $\tau(T_v T_w^\vee)=\delta_{v,w}$ by \eqref{eq:T^vee_w},
we have $\omega_i(e)=P_J^{-1} \sum_{v \in W_J} \chi_i(T_v)$.
Since the module $V_i$ lies in the induced module $e_J \bfH$
on which $T_v$ acts by the index representation \eqref{eq:ind},
we have $\sum_{v \in W_J} \chi_i(T_v) = \sum_{v \in W_J}a_v =P_J$.
Thus we have the result.

\item 
Denote by $e_i:=e_{V_i} \in Z(\bfH)$ the idempotent for $V_i$.
By \eqref{eq:chi^*} we have $\chi_i^* = c_i e_i$.
Then the idempotence of $e_J$ and \eqref{eq:chi^*-orth} yields the conclusion.

\item
We continue to use the symbol $e_i$.
Since $e_i \bfH \simeq V_i$ as $\bfH$-modules, 
we have $\omega_i = c_i e_i \chi_i^* \in V_i$.
\eqref{i:prp:H-zph:2} yields $c_i \omega_i = e_J (\chi_i^*)^2 e_J$,
and since $c_i$ is non-zero as explained in the paragraph of \eqref{eq:e_V},
we have $\omega_i \in e_J \bfH e_J$.
Thus we have the statement.

\item
It follows from \eqref{i:prp:H-zph:3} 
that $\omega_i$'s are linearly independent in $e_J \bfH$.
On the other hand, Schur's lemma yields 
$\dim_{\bfK} e_J \bfH e_J =\dim_{\bfK} \End_{\bfH}(e_J \bfH)=r$.
Thus we have the conclusion.
\end{enumerate}
\end{proof}

For the zonal spherical function $\omega_i \in \bfH$ and $w \in W$, 
we define $\omega_i(T_w) \in \bfK$ by the expansion
\begin{align}
 \omega_i = \sum_{w \in W} \omega_i(T_w) T_w^\vee.
\end{align}
The limit $\rst{\omega_i(T_w)}{q=1}$ coincides with the value of 
the zonal spherical function on $(W,W_J)$ at the double coset 
$W_J w W_J \in W_J \bs W/W_J$.
Unfortunately, $\omega_i(T_w)$ is not constant on the double coset $W_J w W_J$,
and at this point we have a discrepancy between Hecke- and classical settings.

As a closing remark, let us recall that there are some explicit formulas of the values of 
zonal spherical functions for the Gelfand pairs of $\frS_n$ and its subgroups,
and for the pairs of $\GL(n,\bbF_q)$ and its parabolic subgroups,
in terms of terminating hypergeometric and basic hypergeometric series, respectively.
See \cite{D} and \cite[Chap.~VII]{M} for the detail.
However, in the case of our zonal spherical functions for Hecke algebras,
we don't have such nice formulas at this moment.
We invite the reader to tackle this problem.

\begin{Ack}
The author thanks M.~Hayashi and A.~Hora for the fascinating collaboration in \cite{HHY}.
He also thanks Masatoshi Noumi for valuable comments on ``quantized" spherical functions
and hypergeometric functions.
\end{Ack}


\end{document}